\documentclass[12pt,twoside]{amsart}

\usepackage{amssymb}

\theoremstyle{plain}
\newtheorem{theorem}{Theorem}[section]
\newtheorem{lemma}[theorem]{Lemma}
\newtheorem{coro}[theorem]{Corollary}
\newtheorem{prop}[theorem]{Proposition} 

\theoremstyle{remark} 
\newtheorem{rem}[theorem]{Remark}

\newcommand{\ackn}{\noindent{\sc Acknowledgement }\hspace{5pt} }

\newcommand{\F}{\mathbb F}

\DeclareMathOperator{\Soc}{Soc}
\DeclareMathOperator{\id}{id}
\DeclareMathOperator{\Id}{I}
\DeclareMathOperator{\Aut}{Aut}
\DeclareMathOperator{\GL}{GL}
\DeclareMathOperator{\AGL}{A \Gamma L}

\renewcommand{\phi}{\varphi}

\begin{document}

\title[Groups with a base property]{Groups with a base property
  analogous \\ to that of vector spaces} 

%\thanks{Preprint version: \today}

\author{Paul Apisa}

\address{Department of Mathematics, University of Chicago, Chicago, IL
  60615} \email{paul.apisa@gmail.com}

\author{Benjamin Klopsch}

\address{Department of Mathematics, Royal Holloway, University of
  London, Egham TW20 0EX, UK}

\curraddr{Institut f\"ur Algebra und Geometrie, Mathematische 
  Fakult\"at, Otto-von-Guericke-Universit\"at Magdeburg, 39016
  Magdeburg, Germany}

\email{Benjamin.Klopsch@rhul.ac.uk}

%%%

\begin{abstract}
  A $\mathcal{B}$-group is a group such that all its minimal
  generating sets (with respect to inclusion) have the same size.  We
  prove that the class of finite $\mathcal{B}$\nobreakdash-groups is
  closed under taking quotients and that every finite
  $\mathcal{B}$-group is solvable.  Via a complete classification of
  Frattini-free finite $\mathcal{B}$-groups we obtain a general
  structure theorem for finite $\mathcal{B}$-groups.  Applications
  include new proofs for the characterization of finite matroid groups
  and the classification of finite groups with the basis property.
\end{abstract}

\maketitle

%%%%%

\section{Introduction}
 
Let $G$ be a finite group.  A generating set $X$ of $G$ is said to be
\emph{minimal} if no proper subset of $X$ generates~$G$.  We denote by
$d(G)$ the minimal number of generators of $G$, i.e., the smallest
size of a minimal generating set of~$G$, and we write $m(G)$ for the
largest size of a minimal generating set of~$G$.

Whereas the invariant $d(G)$ has been well studied for many groups
$G$, its counterpart $m(G)$ has not received a similar degree of
attention.  First steps toward investigating the latter have been
taken in the context of permutation groups.  For instance,
in~\cite{Wh00} Whiston proved that $m(\textup{Sym}(n)) = n-1$ for the
finite symmetric group of degree~$n$.  Furthermore, Cameron and Cara
gave in~\cite{CaCa02} a complete description of the maximal
independent generating sets of~$\textup{Sym}(n)$; these are precisely
the minimal generating sets of maximal size.  Clearly, Whiston's
result implies that
\[
m(\textup{Sym}(n)) - d(\textup{Sym}(n)) \to \infty \qquad \text{as $n
  \to \infty$.}
\]
This suggests a natural `classification problem': given a non-negative
integer~$c$, characterize all finite groups $G$ such that $m(G)-d(G)
\leq c$.  The results of Saxl and Whiston in~\cite{SaWh02} show that
for projective special linear groups $G = \mathrm{PSL}_2(p^r)$ the
difference $m(G)-d(G)$ depends on the number of prime divisors of~$r$.
In particular, $m(G)-d(G) = 1$ for all $G = \mathrm{PSL}_2(p)$ with
$p$ not congruent to $\pm 1$ modulo $8$ or~$10$.  Since the Frattini
subgroup $\Phi(G)$ consists of all `non-generators' of $G$, we have
$d(G) = d(G/\Phi(G))$ and $m(G) = m(G/\Phi(G))$.  Hence one may
initially focus on groups $G$ which are \emph{Frattini-free}, i.e.,
for which $\Phi(G) = 1$.

In the present article we solve the above stated problem for $c=0$.
We say that the group $G$ has property~$\mathcal{B}$, or the
\emph{weak basis property}, if all its minimal generating sets have
the same size, equivalently if $m(G) = d(G)$.  Groups with property
$\mathcal{B}$ are called $\mathcal{B}$-groups for short.  A group is
said to have the \emph{basis property} if all its subgroups have
property~$\mathcal{B}$.  The Burnside basis theorem states that all
finite $p$-groups are $\mathcal{B}$-groups and, consequently, have the
basis property.

Groups with the basis property as well as variants, such as matroid
groups, have been considered by a number of authors; e.g.,
see~\cite{MBQ} and references therein.  Indeed, McDougall-Bagnall and
Quick initiated in~\cite{MBQ} the systematic study of finite
$\mathcal{B}$-groups and used this to classify groups with the basis
property.  In this context they raised the following fundamental
questions.  Is it true that property $\mathcal{B}$ is inherited by
quotient groups?  Is it possibly true that every finite
$\mathcal{B}$-group is solvable?  We answer both questions positively.

\begin{prop} \label{pro:quotients} Every quotient of a finite
  $\mathcal{B}$-group is again a $\mathcal{B}$-group.
\end{prop}

\begin{theorem} \label{thm:solvable}
  Every finite $\mathcal{B}$-group is solvable.
\end{theorem}

From these structural results we obtain a characterization of finite
$\mathcal{B}$-groups, based on a complete classification of
Frattini-free finite $\mathcal{B}$-groups.  For any prime $p$ we
denote by $\F_p$ the field with $p$ elements.

\begin{theorem} \label{thm:Frattini-free} 
  Let $G$ be a finite group.  Then $G$ is a Frattini-free
  $\mathcal{B}$-group if and only if one of the following holds:
  \begin{enumerate}
  \item $G$ is an elementary abelian $p$-group for some prime $p$;
  \item $G = P \rtimes Q$, where $P$ is an elementary abelian
    $p$-group and $Q$ is a non-trivial cyclic $q$-group, for distinct
    primes $p \ne q$, such that $Q$ acts faithfully on $P$ and the
    $\F_p Q$-module $P$ is a direct sum of isomorphic copies of one
    simple module.
  \end{enumerate}
\end{theorem}

\begin{rem} \label{rem:explicit} This means that there are no
  Frattini-free finite $\mathcal{B}$-groups beyond the examples
  constructed in~\cite[\S3]{MBQ}.  Indeed, the groups listed in (2) of
  Theorem~\ref{thm:Frattini-free} can be concretely realized as
  `semidirect products via multiplication in finite fields of
  characteristic~$p$': the simple module in question is of the form
  $\F_p(\zeta)$, the additive group of a finite field generated by a
  $q^k$th root of unity $\zeta$ over~$\F_p$, with a generator $z$ of
  $Q$ acting on $\F_p(\zeta)$ as multiplication by~$\zeta$.
\end{rem}

Using the explicit description of Frattini-free finite
$\mathcal{B}$-groups, we determine the automorphism groups of such
groups; see Theorem~\ref{thm:automorphism-group}.  From
McDougall-Bagnall and Quick's results in~\cite{MBQ} we obtain a
characterization of finite $\mathcal{B}$-groups.
 
\begin{theorem} \label{thm:finite-B-group}
  Let $G$ be a finite group.  Then $G$ is a $\mathcal{B}$-group if and
  only if one of the following holds:
  \begin{enumerate}
  \item $G$ is a $p$-group for some prime $p$;
  \item $G = P \rtimes Q$, where $P$ is a $p$-group and $Q$ is a
    cyclic $q$-group for, distinct primes $p \ne q$, such that $C_Q(P)
    \ne Q$ and every non-trivial element of $Q/C_Q(P)$ acts
    fixed-point-freely on $P/\Phi(P)$.
  \end{enumerate}
  Moreover, in case $\textup{(2)}$ one has $\Phi(G) = \Phi(P) \times C_Q(P)$.
\end{theorem}
 
As applications of Theorems~\ref{thm:Frattini-free}
and~\ref{thm:Frattini-free} we provide new, streamlined proofs for the
characterization of finite matroid groups (cf.~\cite{SV1}) and the
classification of finite groups with the basis property
(cf.~\cite{MBQ}).  Furthermore, we record as
Corollary~\ref{cor:m-at-most-2} a description of finite groups $G$
with $m(G) \leq 2$.

The proofs of our main results rely ultimately on consequences of the
Classification of Finite Simple Groups.  These enter our proofs
directly as well as indirectly, namely via results of Lucchini and
Menegazzo on generation properties of finite groups with a unique
minimal normal subgroup; see \cite{L1} and \cite{L2}.

\medskip

An outline of the paper is as follows.
Proposition~\ref{pro:quotients} and Theorem~\ref{thm:solvable} are
proved in Section~\ref{sec:2}.  Theorems~\ref{thm:Frattini-free} and
\ref{thm:finite-B-group} are proved in
Section~\ref{sec:3}. Theorem~\ref{thm:automorphism-group} in
Section~\ref{sec:4} describes the automorphism group of a
Frattini-free $\mathcal{B}$-group.  In Sections~\ref{sec:5}
and~\ref{sec:6} we use our main results to derive a characterization
of finite matroid groups and a classification of finite groups with
the basis property.

%%%%%
 
\section{Quotients and solvability of
  $\mathcal{B}$-groups} \label{sec:2}

%B if and only if QB

\begin{proof}[Proof of Proposition~\ref{pro:quotients}]
  Let $G$ be a finite $\mathcal{B}$-group with normal subgroup~$N$,
  and let $\pi \colon G \rightarrow G/N$ denote the projection
  homomorphism.  Writing $d = d(G/N)$, we choose $x_1, \ldots, x_d \in
  G$ such that $x_1^\pi, \ldots, x_d^\pi$ is a minimal generating
  sequence of~$G/N$.

  For a contradiction, assume that $G/N$ does not have
  property~$\mathcal{B}$: let $\bar y_1, \ldots, \bar y_e$ be a
  minimal generating sequence of $G/N$ with $e>d$.  Express each
  element $\bar y_i$ as a word in $x_1^\pi , \ldots, x_d^\pi$ and then
  let $y_i$ denote the same word in $x_1, \ldots, x_d$ so that
  $y_i^\pi = \bar y_i$.  Since $\langle y_1^\pi, \ldots, y_e^\pi
  \rangle = G/N$, there is a sequence $z_1, \ldots, z_f$ in $N$ such
  that $y_1, \ldots, y_e, z_1, \ldots, z_f$ minimally generates~$G$.
  However, the strictly shorter sequence $x_1, \ldots, x_d, z_1,
  \ldots, z_f$ also generates $G$, contradicting the fact that $G$ is
  a $\mathcal{B}$-group.
\end{proof}

\begin{prop} \label{pro:cyclic-B} Let $G$ be a finite group.
  \begin{enumerate}
  \item Suppose that $G$ is simple.  Then $G$ has
    property~$\mathcal{B}$ if and only if $G$ is cyclic.
  \item Suppose that $G$ is cyclic.  Then $G$ has
    property~$\mathcal{B}$ if and only if $G$ has prime-power order.
  \end{enumerate}
\end{prop}

\begin{proof}
  (1) Suppose that $G$ is a non-abelian simple group.  Then the
  classification of finite simple groups implies that $d(G) = 2$
  whereas $m(G) \geq 3$.  The latter follows, for instance, from the
  fact that $G$ is generated by involutions.

  (2) As $G$ is cyclic, $d(G) = 1$ and the primary decomposition of $G$
  shows that $m(G)$ is equal to the number of primes dividing~$\lvert
  G \rvert$.
\end{proof}

The next three lemmas follow from results of Lucchini and Menegazzo in
\cite{L1} and~\cite{L2}.  Their work relies on the Classification of
Finite Simple Groups.

% BFF Groups have d(G) = d(G/N) + 1
\begin{lemma} \label{lem:plus-one}
  Let $G$ be a finite $\mathcal{B}$-group with a minimal normal
  subgroup~$N$.  Then 
  \begin{equation*} d(G) =
    \begin{cases}
      d(G/N) & \text{if $N \leq \Phi(G)$,} \\
      d(G/N) + 1 & \text{otherwise.}
    \end{cases}
  \end{equation*}
\end{lemma}

\begin{proof}
  If $d(G) = d(G/N)$ then every minimal generating sequence of $G$
  projects to a minimal generating sequence of $G/N$; so elements of
  $N$ never appear in a minimal generating sequence of $G$, that is $N
  \leq \Phi(G)$.  Conversely, if $N \leq \Phi(G)$, then $d(G) =
  d(G/N)$.  On the other hand, if $d(G) > d(G/N)$, then $d(G) = d(G/N)
  + 1$ because $d(G) \leq d(G/N) + 1$ from~\cite{L1}.
\end{proof}

\begin{lemma} \label{lem:d-equals-2}
  Let $G$ be a non-cyclic, Frattini-free finite $\mathcal{B}$-group
  with a unique minimal normal subgroup~$N$.  Then $d(G) = 2$ and
  $G/N$ is cyclic of prime-power order.
\end{lemma}

\begin{proof}
  From~\cite{L2} we have $d(G) = \max \{2, d(G/N)\}$ and
  Lemma~\ref{lem:plus-one} yields $d(G) = d(G/N) + 1$.  Thus $d(G) =
  2$ and $d(G/N) = 1$ so that $G/N$ is cyclic.  Moreover,
  Propositions~\ref{pro:quotients} and~\ref{pro:cyclic-B} imply that
  $G/N$ has prime-power order.
\end{proof}

\begin{lemma} \label{lem:d-geq-3} Let $G$ be a Frattini-free finite
  $\mathcal{B}$-group with a non-abelian minimal normal subgroup~$N$.
  If $G/N$ is cyclic, then $m(G) \geq 3$.
\end{lemma}

\begin{proof}
  Suppose that $G/N$ is cyclic.  Then, by
  Propositions~\ref{pro:quotients} and~\ref{pro:cyclic-B}, the
  quotient $G/N$ is cyclic of $p$-power order for some prime~$p$.  Let
  $P$ be a Sylow-$p$ subgroup of~$G$.  Since $G$ cannot be a
  $p$-group, we find a maximal subgroup $H$ of $G$ which contains~$P$.
  Let $q$ be a prime dividing $[G : H]$.

  Let $Q$ be a Sylow-$q$ subgroup contained in $N$, and observe that
  $Q$ is also a Sylow\nobreakdash-$q$ subgroup of~$G$.  From $Q \ne N$
  we conclude that $N_G(Q) \ne G$.  Furthermore, the Frattini argument
  yields $G = N_G(Q) N$.  Since $G/N$ is cyclic of $p$-power order, we
  find an element $g \in N_G(Q)$ of $p$-power order such that $\langle
  g \rangle N = G$.  Moreover, replacing $g$ and $Q$ by conjugates
  $g^x$ and $Q^x$ for a suitable $x \in G$, we may assume without loss
  of generality that $g \in P \leq H$.

  Then $\langle Q \cup H \rangle = G$ and $\langle (H \cap N) \cup
  \{g\} \rangle = H$ so that
  \[
  \langle Q \cup (H \cap N) \cup \{ g \} \rangle = G.
  \]
  We choose a minimal generating set $X$ of $G$ with $X \subseteq Q
  \cup (H \cap N) \cup \{ g \}$.  Since
  \[
  \langle Q \cup (H \cap N) \rangle \subseteq N, \quad \langle Q \cup
  \{g\} \rangle \subseteq N_G(Q), \quad \langle (H \cap N) \cup \{ g
  \} \rangle \subseteq H
  \]
  are all properly contained in $G$, we conclude that $m(G) \geq
  \lvert X \rvert \geq 3$.
\end{proof}

\begin{proof}[Proof of Theorem~\ref{thm:solvable}]
  Let $G$ be a finite $\mathcal{B}$-group.  By
  Proposition~\ref{pro:quotients} every quotient of $G$ has
  property~$\mathcal{B}$ and thus, by induction on the order, every
  proper quotient of $G$ is solvable.

  For a contradiction, assume that $G$ is not solvable and
  consequently has no non-trivial solvable normal subgroups.  In
  particular, since $\Phi(G)$ is nilpotent, this implies that $G$ is
  Frattini-free.  Let $M$ be a minimal normal subgroup of~$G$.  Then
  $M$ is non-abelian and $G$ has no other minimal normal subgroup
  $\tilde M$ besides $M$; otherwise $M$ would embed into the solvable
  group~$G/\tilde M$.

  Hence Lemmas~\ref{lem:d-equals-2} and \ref{lem:d-geq-3} yield the
  contradiction $2 = d(G) = m(G) \geq 3$.
\end{proof}

%%%%%

\section{The classification of $\mathcal{B}$-groups} \label{sec:3}

Recall that the socle $\Soc(G)$ of a finite group $G$ is the subgroup
generated by all minimal normal subgroups.

%Lemma that Tells Us About the socle
\begin{lemma} \label{lem:complement} Let $G$ be a Frattini-free finite
  $\mathcal{B}$-group.  Then $G = S \rtimes K$, where $S = \Soc(G)$ is
  elementary abelian and $C_K(S)$ is trivial.
\end{lemma}

\begin{proof}
  By Theorem~\ref{thm:solvable}, the group $G$ is solvable.  Hence $S$
  is abelian.  We recall that every abelian normal subgroup of a
  Frattini-free finite group admits a complement; e.g., see
  \cite[Proposition~5.2.13]{Ro82}.  Hence $G = S \rtimes K$ for a
  suitable subgroup~$K$.

  Now assume for a contradiction that $S$ is not elementary abelian.
  Let $P$ be a nontrivial Sylow-$p$ subgroup of $S$ and $Q$ a
  nontrivial Sylow-$q$ subgroup of $S$, for distinct primes $p \ne q$.
  Let $L$ be a complement for $P \times Q$ in~$G$ so that
  \[
  G =(P \times Q) \rtimes L.
  \]
  Choose a minimal generating sequence $z_1, \ldots, z_f$ of $L$ and
  extend this to a minimal generating sequence
  \[
  x_1, \ldots, x_d, \, y_1, \ldots, y_e, \, z_1, \ldots, z_f
  \]
  of $G$ of length $d+e+f$ by choosing a minimal generating sequence
  $x_1, \ldots, x_d$ of $P$ as a $\F_p L$-module and a minimal
  generating sequence $y_1, \ldots, y_e$ of $Q$ as a $\F_q L$-module.
  Since $P$ and $Q$ are non-trivial, the parameters $d$ and $e$ are
  positive.  But then
  \[
  x_1 y_1, \, x_2 \ldots, x_d, \, y_2, \ldots, y_e, \, z_1, \ldots,
  z_f
  \]
  is a generating sequence of $G$ of shorter length $d+e+f-1$,
  contradicting the fact that $G$ is a $\mathcal{B}$-group.

  Finally, $C_K(S)$ is invariant under conjugation by $S$ and by $K$.
  Hence it is a normal subgroup of $G$ which intersects the socle $S$
  trivially, and $C_K(S) = 1$.
\end{proof}

\begin{proof}[Proof of Theorem~\ref{thm:Frattini-free}]
  It is shown in \cite[\S3]{MBQ} that groups of the form specified in the
  theorem are Frattini-free and have property~$\mathcal{B}$.

  Now suppose that $G$ is a Frattini-free $\mathcal{B}$-group. By
  Lemma~\ref{lem:complement}, the group $G$ is abelian if and only if
  $G$ is elementary abelian.  Thus it suffices to analyse the
  situation where $G$ is non-abelian.  By Lemma~\ref{lem:complement},
  we have $G = P \rtimes Q$, where $P = \Soc(G)$ is an elementary
  abelian $p$-group for a prime $p$ and the complement $Q \ne 1$ acts
  faithfully on~$P$.  We decompose $P$ as a direct product $P = M_1
  \times \ldots \times M_d$, where each $M_i$ is a minimal normal
  subgroup of $G$, that is each $M_i$ is a simple $\F_pQ$-module.

  Fix a minimal generating sequence $z_1, \ldots, z_e$ of~$Q$.  By
  choosing in each $M_i$ an element $x_i \ne 1$, we obtain a minimal
  generating sequence $x_1, \ldots, x_d, z_1, \ldots, z_e$ of $G$ of
  length $d+e$.  If $M_i$ and $M_j$ were not isomorphic as $\F_p
  Q$-modules for some $i \ne j$ we could replace $x_i$ and $x_j$ by a
  single element $x_ix_j$ to obtain a minimal generating sequence of
  $G$ of length $d+e-1$.  Since $G$ has property $\mathcal{B}$ this
  cannot happen and thus all the $M_i$ are isomorphic to one another
  as $\F_p Q$-modules.  In particular, this implies that $Q$ acts
  faithfully on each~$M_i$.

  Finally we show that $Q$ is a cyclic $q$-group for some prime $q \ne
  p$.  For this we consider the quotient group
  \[
  \bar G = G/(M_2 \times \ldots \times M_d) = \bar M_1 \rtimes \bar Q
  \]
  which has property~$\mathcal{B}$ by Proposition~\ref{pro:quotients}.
  Clearly, $\bar M_1$ is an abelian minimal normal subgroup of $\bar
  G$ and thus $\bar Q$ a maximal subgroup of $\bar G$.  We claim that
  $\bar M_1$ is the unique minimal normal subgroup of $\bar G$.
  Indeed, if $\bar g \in \bar G \setminus \bar M_1$ then $g \in G
  \setminus P$ and $\langle g \rangle^G \supseteq [M_1,g]^G = M_1$;
  consequently, every normal subgroup of $\bar G$ contains~$\bar M_1$.
  Moreover, from $\bar M_1 \cap \bar Q = 1$ we conclude that $\bar G$
  is Frattini-free.  Lemma~\ref{lem:d-equals-2} shows that $Q \cong
  \bar Q$ is a cyclic $q$-group for some prime~$q$.  Since $G$ is
  non-abelian and Frattini-free, it cannot be a $p$-group, and hence
  $q \ne p$.
\end{proof}

\begin{proof}[Proof of Theorem~\ref{thm:finite-B-group}]
  First suppose that $G$ has property $\mathcal{B}$.  From
  Proposition~\ref{pro:quotients} we deduce that $H = G/\Phi(G)$ is a
  Frattini-free $\mathcal{B}$-group so that
  Theorem~\ref{thm:Frattini-free} gives a detailed description of~$H$.
  We observe that if $H$ is a $p$-group for some prime~$p$ then $G$ is
  a $p$-group, because $H$ is the image of any Sylow-$p$ subgroup of
  $G$ modulo $\Phi(G)$ and $\Phi(G)$ consists of the `non-generators'
  of~$G$.  Now suppose that $H = P \rtimes Q$ with $P$ and $Q$ as in
  (2) of Theorem~\ref{thm:Frattini-free}.

  We claim that every non-trivial element of $Q$ acts
  fixed-point-freely on~$P$, i.e., $C_P(y) = 1$ for $y \in Q \setminus
  \{1\}$.  Recall that $P$ is a direct sum of isomorphic copies of one
  simple $\F_p Q$-module~$M$.  Thus $Q$ acts faithfully on $M$ and,
  since $Q$ is abelian, this implies that every non-trivial element of
  $Q$ acts fixed-point-freely on~$M$ and therefore also on~$P$.  From
  \cite[Proposition~3.3 and Theorem~3.4]{MBQ} we deduce that $G$ is of
  the shape described in the theorem.

  Conversely, if $G$ has prime-power order then $G$ is a
  $\mathcal{B}$-group by the Burnside basis theorem, and it suffices
  to consider the remaining case: $G \cong P \rtimes Q$, where $P$ is
  a $p$-group and $Q$ is a cyclic $q$-group, for distinct primes $p
  \ne q$, such that $C_Q(P) \ne Q$ and every non-trivial element of
  $Q/C_Q(P)$ acts fixed-point-freely on $P/\Phi(P)$.  Since $P$ is a
  $p$-group and since $C_Q(P)$ is a proper subgroup of the cyclic
  $q$-group~$Q$, each subset of $G$ that generates $G$ modulo the
  normal subgroup $\Phi(P) \times C_Q(P)$ already generates~$G$.  Thus
  $\Phi(P) \times C_Q(P) \subseteq \Phi(G)$ and in order to show that
  $G$ has property $\mathcal{B}$ we may assume that $\Phi(P) \times
  C_Q(P) = 1$.  Then $P$ is elementary abelian and every element of
  $Q$ acts fixed-point-freely on $P$.  It follows from
  \cite[Proposition~3.3 and Theorem~3.2]{MBQ} that $G$ has
  property~$\mathcal{B}$.
\end{proof}

As a consequence of Theorem~\ref{thm:Frattini-free},
Remark~\ref{rem:explicit} and Theorem~\ref{thm:finite-B-group} we
record the following corollary.

\begin{coro} \label{cor:m-at-most-2}
  Let $G$ be a non-trivial finite group with $m(G) \leq 2$.  Then
  precisely one of the following holds:
  \begin{enumerate}
  \item $G$ is cyclic of prime-power order so that $d(G) = m(G) =1$;
  \item $G$ is cyclic of order divisible by exactly two distinct
    primes so that $d(G) = 1$ and $m(G)=2$;
  \item $G$ is a group of prime-power order with $d(G) = m(G) = 2$;
  \item $G$ is a $\mathcal{B}$-group of order divisible by exactly two
    distinct primes so that $d(G) = m(G) = 2$.
  \end{enumerate}
  In the last case $G/\Phi(G) \cong P \rtimes Q$, where the elementary
  abelian $p$-group $P$ is isomorphic to the additive group of a
  finite field~$F = \F_p(\zeta)$, the cyclic $q$-group $Q = \langle z
  \rangle$ embeds as $\langle \zeta \rangle$ into $F^\times$ and $Q$
  acts on $P$ via multiplication in $F$.
\end{coro}

%%%%%%%%%%%%%%%%%%%%%%%%%%%%%
%Section 4: Classifying Groups with B
%%%%%%%%%%%%%%%%%%%%%%%%%%%%%
%
%\section{Two Excised Lemmas}
%
%\begin{lemma}
%  Let $G$ be a finite group containing a normal subgroup $N$ of prime
%  index $p$ (e.g. a noncyclic B group).  Then there exists an element
%  $x \in G \backslash N$ contained in more than one maximal subgroup
%  of $G$ if and only if $m(G) \geq 3$.
%\end{lemma}
%
%\begin{lemma}
%  Let $G$ be a finite nonnilpotent group with normal subgroup $N$ of
%  prime index $p$.
%  \begin{enumerate}
%  \item If $m(G) \leq 2$, then all maximal subgroups, excluding $N$,
%    are conjugate
%  \item If $G$ contains no normal p-subgroups, then $m(G) \geq 3$
%  \end{enumerate}
%\end{lemma}
%
%%%%%%%%%%%%%%%%%%%%%%%%%%%%%
%Section 5: Consequences of B Group Classification
%%%%%%%%%%%%%%%%%%%%%%%%%%%%%

\section{Automorphisms of Frattini-free
  $\mathcal{B}$-groups} \label{sec:4}

Let $G$ be a Frattini-free finite $\mathcal{B}$-group.  By
Theorem~\ref{thm:Frattini-free} we have $\lvert G \rvert = p^{rd}
q^k$, where $p \ne q$ are distinct primes and the socle $\Soc(G)$ is a
direct product of $d$ minimal normal subgroups, each elementary
abelian of size~$p^r$.  Moreover, as indicated in
Remark~\ref{rem:explicit}, the group $G$ can be embedded into the
general semi-affine group $\AGL(d,F)$ of degree $d$ over the finite
field $F = \F_p(\zeta)$, where $\zeta$ denotes a primitive $q^k$th
root of unity over~$\F_p$ and $[F:\F_p] = r$: writing
\[
\AGL(d,F) = F^d \rtimes \GL(d,F) \rtimes \Aut(F)
\]
we can realize $G$ as the subgroup consisting of all elements of the form
\[
\left( v, \zeta^n \Id, \id_F \right), \quad \text{where $v \in F^d$ and $\zeta^n
    \Id \in \GL(d,F)$ is scalar for $0 \leq n < q^k$.}
\] 

\begin{theorem} \label{thm:automorphism-group} Let $G \leq \AGL(d,F)$
  be a Frattini-free finite $\mathcal{B}$-group, embedded into the
  general semi-affine group of degree $d$ over a finite field $F$ as
  above.  Then $G$ is normal in $\AGL(d,F)$ and the action of
  $\AGL(d,F)$ on $G$ by conjugation induces the full automorphism
  group of~$G$.  Moreover, the action of $\AGL(d,F)$ on $G$ is
  faithful, unless $G$ is abelian.
\end{theorem}

\begin{proof}
  Clearly, $G$ is a normal subgroup of $\AGL(d,F)$.  If $G$ is
  abelian, then $r=1$, $k=0$ and $F = \F_p$ so that $G = \F_p^d$ and
  $\AGL(d,F) = \F_p^d \rtimes \GL(d,\F_p)$.  In this case $\Aut(G)
  \cong \GL(d,\F_p)$ is realized by the action of $\AGL(d,\F_p)$
  modulo $\F_p^d$.

  Now suppose that $G$ is non-abelian.  Since $F^d \subseteq G$, the
  centralizer of $G$ in $\AGL(d,F)$ is contained in~$F^d$.  Since
  $\zeta \Id \in G$ acts fixed-point-freely on~$F^d$, the centralizer
  of $G$ in $\AGL(d,F)$ is trivial. Hence the action of $\AGL(d,F)$ on
  $G$ by conjugation is faithful.  It remains to show that every
  automorphism of $G$ can be realized as conjugation by a suitable
  element of~$\AGL(d,F)$.

  Let $\phi \in \Aut(G)$.  We observe that $\Soc(G) = F^d$ and that
  the action of $G$ on $\Soc(G)$ by conjugation induces an embedding
  of $G/ \Soc(G)$ into $F^\times$.  The element $\zeta_2 \in F^\times$
  corresponding to the action of $(\zeta \Id)^\phi$ on $F^d$ satisfies
  the same minimal polynomial over $\F_p$ as~$\zeta$.  Thus the action
  of $\phi$ on $G/\Soc(G)$ can also be realized by an element of
  $\Aut(F) \leq \AGL(d,F)$.  Without loss of generality we may
  therefore assume that $\phi$ acts as the identity on $G/\Soc(G)$.
  Then $\phi$ acts on $\Soc(G) = F^d$ as an $F$-linear isomorphism.  A
  suitable element of $\GL(d,F) \leq \AGL(d,F)$ realizes the same
  action and we may further assume that $\phi$ restricts to the
  identity on $\Soc(G)$.  This means that $(\zeta \Id)^\phi = v (\zeta
  \Id)$ for some $v \in F^d$.  Finally, we notice that conjugation by
  $(\zeta-1)^{-1} v \in F^d \leq \AGL(d,F)$ induces the automorphism
  $\phi$.
\end{proof}

%SECTION 5: Matroids
\section{The characterization of matroid groups} \label{sec:5}

A subset $X$ of a finite group $G$ is called
\emph{independent},  respectively \emph{Frattini-independent}, if there
is no proper subset $Y \subset X$ such that $\langle X \rangle =
\langle Y \rangle$, respectively $\langle X \cup \Phi(G) \rangle =
\langle Y \cup \Phi(G) \rangle$.  The group $G$ is called a
\emph{matroid group} if $G$ has property~$\mathcal{B}$ and every
Frattini-independent subset of $G$ can be extended to a minimal
generating set of~$G$.  Alternatively, $G$ is a matroid group if $H =
G/\Phi(G)$ is a Frattini-free $\mathcal{B}$-group and every
independent subset of $H$ can be extended to an minimal generating
set.  The definition of a matroid group given here is the one used in~\cite{SV1,SV2}. 

% \begin{prop}
%   Let $G = \K^n \rtimes Q$ be a nonabelian BFF group. A sequence $x$
%   is a minimal generating sequence if and only if it contains an
%   element of the form $a = (v_0, \sigma)$, where $\sigma$ generates
%   $Q$, and all other elements of $x$ are $v_i a^{n_i}$ for some
%   integers $n_i$ and where $<v_1,...,v_n>$ is linearly independent
%   over $\K$.
% \end{prop}

We obtain a small variation of the characterization of matroid groups
in~\cite{SV1}.

\begin{theorem}[Scapellato and Verardi \cite{SV1}]
  Let $G$ be a finite group and let $H = G/\Phi(G)$. The group $G$ is
  a matroid group if and only if one of the following holds:
  \begin{enumerate}
  \item $G$ is a $p$-group for some prime $p$,
  \item $H = P \rtimes Q$, where $P \cong \F_p^d$ and $Q$ is cyclic of
    order $q$, for primes $p, q$ such that $q \mid p - 1$, and $Q
    \hookrightarrow \F_p^\times$ acts on $P$ via field multiplication.
  \end{enumerate}
\end{theorem}

\begin{proof}
  By the Burnside basis theorem every finite group of prime-power
  order is a matroid group.  From now suppose that $G$ does not have
  prime-power order.

  First suppose that $G$ is a matroid group.  Then, by
  Theorem~\ref{thm:Frattini-free} and Remark~\ref{rem:explicit}, the
  Frattini quotient $H$ is a matroid group of the form $H = P \rtimes
  Q$, where $P$ is an elementary abelian $p$-group and $Q$ is a
  non-trivial cyclic group of order $q^k$, for distinct primes $p \ne
  q$, such that $Q \hookrightarrow F^\times$ acts faithfully on $P
  \cong F^d$ via multiplication in a finite field~$F$.  Here $F$ is
  obtained from $\F_p$ by adjoining a primitive $q^k$th root of unity
  and we set $r = [F:\F_p]$.  We observe that the common size of all
  minimal generating sets of $G$ is~$d+1$.

  Being isomorphic to an $\F_p$-vector space of dimension $rd$, the
  subgroup $P$ contains an independent subset of size~$rd$.  This
  subset extends to a minimal generating set of~$H$.  We deduce that
  $rd \leq d$, thus $r = 1$.  Let $z$ be a generator of $Q$ and assume
  for a contradiction that~$k \geq 2$.  Choose a minimal generating
  set $X$ for $P$ as an $\F_p Q$-module.  Then $X \cup \{ z^q \}$ is
  an independent set of size $d+1$ that does not generate~$H$ and does
  not extend to a minimal generating set of~$H$.  This implies that
  $H$ is not a matroid group in contradiction to our assumptions.
  Hence, $k = 1$, i.e., $Q$ is cyclic of order $q$.  From $Q
  \hookrightarrow \mathbb{F}_p^\times$ we obtain $q \mid p - 1$.

  Conversely, suppose that $H = P \rtimes Q$, where $P \cong \F_p^d$
  and $Q = \langle z \rangle$ is cyclic of order $q$, for primes $p,
  q$ such that $q \mid p - 1$, and $Q \hookrightarrow \F_p^\times$
  acts on $P$ via field multiplication.  By
  Theorem~\ref{thm:Frattini-free} the group $H$ has
  property~$\mathcal{B}$ and it suffices to show that every
  independent subset of $H$ extends to a minimal generating set.  Let
  $X = \{ x_1, \ldots, x_m \} \subseteq H$ be an independent subset of
  size~$m$.  If $X \subseteq P$ then, regarding $P$ as an
  $\F_p$-vector space, we extend $X$ to a minimal generating set of
  $P$ and add the generator $z$ of $Q$ to obtain a minimal generating
  set of~$H$.  Now suppose that $X \not \subseteq P$.  Since $H$ does
  not contain any element of order $pq$, we may assume without loss of
  generality that $x_1 = z$.  Then $X = \{z, v_2 z^{j_2}, \ldots, v_m
  z^{j_m} \}$ where $\{v_2, \ldots, v_m\} \subseteq P$ is an
  independent subset of size $m-1$ and $j_2, \ldots, j_m$ are
  integers.  We extend $\{v_2, \ldots, v_m\}$ to a minimal generating
  set $\{v_2, \ldots, v_d\}$ of~$P$.  Then $X \cup \{v_{m+1}, \ldots,
  v_d\}$ is a minimal generating set of~$H$.
\end{proof}  

Using Theorem~\ref{thm:finite-B-group} we obtain the following
consequence.

\begin{coro}
  Let $G$ be a finite group.  Then $G$ is a matroid group if and only
  if one of the following holds:
  \begin{enumerate}
  \item $G$ is a $p$-group for some prime $p$,
  \item $G = P \rtimes Q$, where $P$ is a $p$-group, $Q$ is a cyclic
    $q$-group for primes $p,q$ such that $q \mid p-1$, $Q/C_Q(P)$ has
    order $q$ and acts on $P/\Phi(P)$ fixed-point-freely.
  \end{enumerate}
\end{coro}

% The Subgroup Lattice of B Groups
\section{The classification of groups with the basis
  property} \label{sec:6}
 
\begin{lemma} \label{lem:no-fix-B} Let $G = P \rtimes Q$, where $P$ is
  a $p$-group and $Q$ a cyclic $q$-group, for distinct primes $p \ne
  q$, such that every non-trivial element of $Q$ acts
  fixed-point-freely on~$P$.  Then $G$ has property~$\mathcal{B}$.
\end{lemma}

\begin{proof}
  If $Q = 1$ then the claim follows from the Burnside basis theorem.
  From now assume that $Q$ is non-trivial.  Clearly, $C_Q(P)=1$ and by
  Theorem~\ref{thm:finite-B-group} it suffices to show that every
  non-trivial element of $Q$ acts fixed-point-freely on~$P/\Phi(P)$.
  Let $y \in Q \setminus \{1\}$ and assume for a contradiction that
  $y$ has a non-zero set of fixed points $U$ in $P/\Phi(P)$.  By
  Maschke's theorem $P/\Phi(P)$ is a semisimple $\F_pQ$-module.  Hence
  there is a submodule $W$ such that $P/\Phi(P) = U \oplus W$.  Then
  $[P, \langle y \rangle]$ is strictly smaller than $P$ because
  $[P/\Phi(P), \langle y \rangle] \leq W$.  By
  \cite[Theorem~5.3.5]{Gor} we have $P = C_P(\langle y \rangle) [P,
  \langle y \rangle ]$, hence $C_P(\langle y \rangle) \ne 1$.  This
  contradicts the fact that $y$ acts on $P$ fixed-point-freely.
\end{proof}
 
\begin{lemma} \label{lem:no-fix-same-shape} Let $G = P \rtimes Q$,
  where $P$ is a $p$-group and $Q$ is a cyclic $q$-group, for distinct
  primes $p \ne q$, such that every non-trivial element of $Q$ acts
  fixed-point-freely on~$P$.  Let $H \leq G$.  Then one of the
  following holds:
  \begin{enumerate}
  \item $H$ is a $p$-group or a $q$-group,
  \item $H$ is conjugate in $G$ to a group of the form $R \rtimes S$
    with $R \leq P$ and $S \leq Q$.
  \end{enumerate}
\end{lemma}

\begin{proof}
  Suppose that $H$ is a subgroup of~$G$ which is not of prime-power
  order.  Then $H \cap P$ and $H/(H \cap P)$ are non-trivial.  Choose
  an element $h \in H$ such that $H = (H \cap P) \langle h \rangle$
  and let $y \in Q$ such that $h \equiv y$ modulo~$P$.  Since $y$ acts
  fixed-point-freely on~$P$, we have $yP = \{ y^x \mid x \in P \}$ so
  that $h = y^x$ for some $x \in P$.  Consequently, $H = (R \rtimes
  S)^x$, where $R = (H \cap P)^{x^{-1}}$ and $S = \langle y \rangle$.
\end{proof}

Recall that a finite group $G$ has the basis property if all its
subgroups are $\mathcal{B}$-groups.
 
\begin{theorem}[McDougall-Bagnall and Quick \cite{MBQ}]
  Let $G$ be a finite group.  Then $G$ has the basis property if and
  only if one of the following holds:
  \begin{enumerate}
  \item $G$ is a $p$-group for some prime $p$,
  \item $G \cong P \rtimes Q$, where $P$ is a $p$-group and $Q$ is a
    non-trivial cyclic $q$-group, for distinct primes $p \ne q$, such
    that every non-trivial element of $Q$ acts fixed-point-freely
    on~$P$.
  \end{enumerate}
\end{theorem}

\begin{proof}
  Lemmas~\ref{lem:no-fix-B} and \ref{lem:no-fix-same-shape} imply that
  every group of the described form has the basis property.
  Conversely, suppose that $G$ has the basis property and is not of
  prime-power order.  From Theorem~\ref{thm:finite-B-group} we deduce
  that $G \cong P \rtimes Q$, where $P$ is a $p$-group and $Q$ is a
  non-trivial cyclic $q$-group.  If there were non-trivial, commuting
  elements $x \in P$ and $z \in Q$ then the cyclic group $\langle xz
  \rangle = \langle x \rangle \rtimes \langle z \rangle$ would not
  have property~$\mathcal{B}$, contradicting the basis property.
  Hence every non-trivial element of $Q$ acts fixed-point-freely
  on~$P$.
\end{proof}

\bigskip

\ackn The first author gratefully acknowledges research support by the
National Science Foundation through the Research Experiences for
Undergraduates (REU) Program at Cornell.  He thanks R.~Guralnick,
R.~K.~Dennis, and D.~Collins.
 
%%%%%

\end{document}